\documentclass[11pt,reqno]{amsart}
\usepackage{color}

\usepackage{amsmath}
\usepackage{amsfonts,amscd}
\usepackage{amssymb}
\usepackage{url}

\usepackage[english]{babel}
\usepackage{booktabs}

\theoremstyle{plain}
\newtheorem{theorem}                 {Theorem}      [section]

\newtheorem{lemma}        [theorem]  {Lemma}
\newtheorem{proposition}  [theorem]  {Proposition}

\theoremstyle{definition}

\newtheorem{definition}   [theorem]  {Definition}

\numberwithin{equation}{section}

\def \cn{{\mathbb C}}
\def \hn{{\mathbb H}}
\def \H{{\mathbb H}}

\def \rn{{\mathbb R}}

\def \zn{{\mathbb Z}}

\def \B{\mathcal B}
\def \E{\mathcal E}

\def \H{\mathcal H}

\def\nab#1#2{\hbox{$\nabla$\kern -.3em\lower 1.0 ex
		\hbox{$#1$}\kern -.1 em {$#2$}}}

\def \Re{\mathfrak R\mathfrak e}

\def \g{\mathfrak{g}}

\def \l{\mathfrak{l}}
\def \m{\mathfrak{m}}

\def \GLR#1{\text{\bf GL}_{#1}(\rn)}

\def \glr#1{\mathfrak{gl}_{#1}(\rn)}
\def \GLC#1{\text{\bf GL}_{#1}(\cn)}
\def \glc#1{\mathfrak{gl}_{#1}(\cn)}
\def \glcp#1{\mathfrak{gl}^+_{#1}(\cn)}
\def \glcm#1{\mathfrak{gl}^-_{#1}(\cn)}
\def \GLH#1{\text{\bf GL}_{#1}(\hn)}
\def \glh#1{\mathfrak{gl}_{#1}(\hn)}
\def \glhp#1{\mathfrak{gl}^+_{#1}(\hn)}
\def \glhm#1{\mathfrak{gl}^-_{#1}(\hn)}

\def \SLR#1{\text{\bf SL}_{#1}(\rn)}
\def \slr#1{\mathfrak{sl}_{#1}(\rn)}
\def \SLC#1{\text{\bf SL}_{#1}(\cn)}
\def \slc#1{\mathfrak{sl}_{#1}(\cn)}
\def \SLH#1{\text{\bf SL}_{#1}(\hn)}
\def \slh#1{\mathfrak{sl}_{#1}(\hn)}

\def \SO#1{\text{\bf SO}(#1)}
\def \so#1{\mathfrak{so}(#1)}
\def \SOs#1{\text{\bf SO}^*(#1)}
\def \sos#1{\mathfrak{so}^*(#1)}
\def \SOO#1#2{\text{\bf SO}(#1,#2)}
\def \soo#1#2{\mathfrak{so}(#1,#2)}
\def \SOC#1{\text{\bf SO}(#1,\cn)}
\def \soC#1{\mathfrak{so}(#1,\cn)}

\def \U#1{\text{\bf U}(#1)}
\def \u#1{\mathfrak{u}(#1)}

\def \SU#1{\text{\bf SU}(#1)}
\def \su#1{\mathfrak{su}(#1)}
\def \SUU#1#2{\text{\bf SU}(#1,#2)}
\def \suu#1#2{\mathfrak{su}(#1,#2)}
\def \SUs#1{\text{\bf SU}^*(#1)}
\def \sus#1{\mathfrak{su}^*(#1)}

\def \Sp#1{\text{\bf Sp}(#1)}
\def \sp#1{\mathfrak{sp}(#1)}
\def \Spp#1#2{\text{\bf Sp}(#1,#2)}
\def \spp#1#2{\mathfrak{sp}(#1,#2)}
\def \SpR#1{\text{\bf Sp}(#1,\rn)}
\def \spR#1{\mathfrak{sp}(#1,\rn)}
\def \SpC#1{\text{\bf Sp}(#1,\cn)}
\def \spC#1{\mathfrak{sp}(#1,\cn)}

\DeclareMathOperator{\Div}{div} 

\DeclareMathOperator{\trace}{trace}

\numberwithin{equation}{section}
\allowdisplaybreaks

\begin{document}

\title[$p$-Harmonic functions and harmonic morphisms on Lie groups]
{Proper $p$-harmonic functions and harmonic morphisms on the classical non-compact semi-Riemannian Lie groups}

\author{Elsa Ghandour}
\address{Mathematics, Faculty of Science\\
University of Lund\\
Box 118, Lund 221 00\\
Sweden}
\email{Elsa.Ghandour@math.lu.se}

\author{Sigmundur Gudmundsson}
\address{Mathematics, Faculty of Science\\
	University of Lund\\
	Box 118, Lund 221 00\\
	Sweden}
\email{Sigmundur.Gudmundsson@math.lu.se}

\begin{abstract}
We apply the method of eigenfamilies to construct new explicit complex-valued $p$-harmonic functions on the  non-compact classical Lie groups, equipped with their natural semi-Riemannian metrics.  We then employ this same approach to manufacture explicit complex-valued harmonic morphisms on these groups.
\end{abstract}

\subjclass[2020]{31B30, 53C43, 58E20}

\keywords{$p$-harmonic functions, harmonic morphisms, semi-Riemannian classical groups}


\maketitle


\section{Introduction}
\label{section-introduction}


For a positive integer $p$, the complex-valued $p$-harmonic functions are solutions to a partial differential equation of order $2\,p$. This equation arises in various contexts, see for example the extensive analysis in \cite{Gaz-Gru-Swe} and a historic account in \cite{Mel}.  The best known applications are in physics e.g. for $p=2$ in the areas of continuum mechanics, including elasticity theory and the solution of Stokes flows. 
The literature on $2$-harmonic functions is vast, but until quiet recently, the domains were either surfaces or open subsets of flat Euclidean space, with only very few exceptions.  For this see the regularly updated online bibliography \cite{Gud-p-bib} maintained by the second author.
\medskip

In their recent article \cite{Gud-Sob-1}, the authors produce  $p$-harmonic functions on the classical Lie groups equipped with their standard Riemannian metrics.  The primary goal of this work is to extend the study to the semi-Riemannian situation.  By Theorem  \ref{theorem-p-harmonic} we show how the problem can be reduced to finding an {\it eigenfamily}, i.e. a collection of complex-valued functions which are eigen both with respect to the {\it Laplace-Beltrami operator} $\tau$ and the {\it conformality operator} $\kappa$, on the semi-Riemannian manifolds involved.  The main part of this paper is devoted to the construction of such families on the following classical Lie groups equipped with their natural semi-Riemannian metrics
\begin{equation*}
\GLC n,\,\GLR n,\,\GLH n,\,\SLC n,\,\SLR n,\,\SLH n,
\end{equation*}
\begin{equation*}
\SOC n,\,\SpC n,\,\SpR n,\,\SOs{2n},\,\SUU pq,\,\SOO pq,\,\Spp pq.
\end{equation*}

Our eigenfamilies can also be used to manufacture complex-valued harmonic morphisms on these manifolds as explained in Theorem \ref{theorem-rational}.  They can therefore be seen as an interesting byproduct of the process presented here.

For semi-Riemannian geometry we recommend O'Neill's standard text  \cite{O-Nei}.  Readers not familiar with harmonic morphisms are advised to consult the standard text \cite{Bai-Woo-book} by Baird and Wood, \cite{Fug-1}, \cite{Fug-2}, \cite{Ish} and the regularly updated online bibliography \cite{Gud-bib}. For the details from Lie theory, used in this paper, we refer the reader to \cite{Hel} and \cite{Kna}.


\section{Eigenfunctions and Eigenfamilies }
\label{section-eigenfamilies}


In this paper we manufacture explicit complex-valued proper $p$-harmonic functions and harmonic morphisms on semi-Riemannian manifolds.  For this we apply two different construction techniques which are presented in  Theorem \ref{theorem-p-harmonic} and in Theorem \ref{theorem-rational}, respectively.  The main ingredients for both these recipes are the common eigenfunctions for the tension field $\tau$ and the conformality operator $\kappa$ which we now describe.
\medskip

Let $(M,g)$ be an $m$-dimensional semi-Riemannian manifold and $T^{\cn}M$ be the complexification of the tangent bundle $TM$ of $M$. We extend the metric $g$ to a complex-bilinear form on $T^{\cn}M$.  Then the gradient $\nabla\phi$ of a complex-valued function $\phi:(M,g)\to\cn$ is a section of $T^{\cn}M$.  In this situation, the well-known complex linear {\it Laplace-Beltrami operator} (alt. {\it tension field}) $\tau$ on $(M,g)$ acts locally on $\phi$ as follows
$$
\tau(\phi)=\Div (\nabla \phi)=\sum_{i,j=1}^m\frac{1}{\sqrt{|g|}} \frac{\partial}{\partial x_j}
\Big(g^{ij}\, \sqrt{|g|}\, \frac{\partial \phi}{\partial x_i}\Big).
$$
For two complex-valued functions $\phi,\psi:(M,g)\to\cn$ we have the following well-known fundamental relation
\begin{equation*}\label{equation-basic}
\tau(\phi\cdot \psi)=\tau(\phi)\cdot\psi +2\cdot\kappa(\phi,\psi)+\phi\cdot\tau(\psi),
\end{equation*}
where the complex bilinear {\it conformality operator} $\kappa$ is given by $$\kappa(\phi,\psi)=g(\nabla \phi,\nabla \psi).$$  Locally this satisfies 
$$\kappa(\phi,\psi)=\sum_{i,j=1}^mg^{ij}\cdot\frac{\partial\phi}{\partial x_i}\frac{\partial \psi}{\partial x_j}.$$

\begin{definition}\cite{Gud-Sak-1}\label{definition-eigenfamily}
Let $(M,g)$ be a semi-Riemannian manifold. Then a complex-valued function $\phi:M\to\cn$ is said to be an {\it eigenfunction} if it is eigen both with respect to the Laplace-Beltrami operator $\tau$ and the conformality operator $\kappa$ i.e. there exist complex numbers $\lambda,\mu\in\cn$ such that $$\tau(\phi)=\lambda\cdot\phi\ \ \text{and}\ \ \kappa(\phi,\phi)=\mu\cdot \phi^2.$$	
A set $\E =\{\phi_i:M\to\cn\,|\,i\in I\}$ of complex-valued functions is said to be an {\it eigenfamily} on $M$ if there exist complex numbers $\lambda,\mu\in\cn$ such that for all $\phi,\psi\in\E$ we have 
$$\tau(\phi)=\lambda\cdot\phi\ \ \text{and}\ \ \kappa(\phi,\psi)=\mu\cdot \phi\cdot\psi.$$ 
\end{definition}

The following Theorem \ref{theorem-poly-families} shows that, given an eigenfamily $\E$, one can employ this to produce an extensive collection $\H^d_\E$ of further such objects.  The result is a semi-Riemannian version of Theorem 2.2 proven in \cite{Gha-Gud-2} for the Riemannian case.

\begin{theorem}\label{theorem-poly-families}
Let $(M,g)$ be a semi-Riemannian manifold and the set of complex-valued functions  $$\E=\{\phi_k:M\to\cn\,|\,k=1,2,\dots,n\}$$ 
be an eigenfamily on $M$ i.e. there exist complex numbers $\lambda,\mu\in\cn$ such that for all $\phi,\psi\in\E$ $$\tau(\phi)=\lambda\cdot\phi\ \ \text{and}\ \ \kappa(\phi,\psi)=\mu\cdot\phi\cdot\psi.$$  
Then the set of complex homogeneous polynomials of degree $d$
$$\H^d_\E=\{P:M\to\cn\,|\, P\in\cn[\phi_1,\phi_2,\dots,\phi_n],\ P(\alpha\cdot\phi)=\alpha^d\cdot P(\phi),\ \alpha\in\cn\}$$ 
is an eigenfamily on $M$ such that for all $P,Q\in\H^d_\E$ we have
$$\tau(P)=(d\cdot\lambda+d(d-1)\cdot\mu)\cdot P\ \ \text{and}\ \ \kappa(P,Q)=d^2\cdot\mu\cdot P\cdot Q.$$
\end{theorem}

\begin{proof}
The statement can be proven with exactly the same arguments as its Riemannian counterpart, see Theorem 2.2 in \cite{Gha-Gud-2}.
\end{proof}


\section{Proper $p$-Harmonic Functions}
\label{section-p-harmonic-functions}


In this section we describe a method for manufacturing  complex-valued proper $p$-harmonic functions on semi-Riemannian manifolds. This method was recently introduced for the Riemannian case in \cite{Gud-Sob-1}.  
\medskip

\begin{definition}\label{definition-proper-p-harmonic}
	Let $(M,g)$ be a semi-Riemannian manifold. For a positive integer $p$, the iterated Laplace-Beltrami operator $\tau^p$ is given by
	$$\tau^{0} (\phi)=\phi\ \ \text{and}\ \ \tau^p (\phi)=\tau(\tau^{(p-1)}(\phi)).$$	We say that a complex-valued function $\phi:(M,g)\to\cn$ is
	\begin{enumerate}
		\item[(i)] {\it $p$-harmonic} if $\tau^p (\phi)=0$ and
		\item[(ii)] {\it proper $p$-harmonic} if $\tau^p (\phi)=0$ and $\tau^{(p-1)}(\phi)$ does not vanish identically.
	\end{enumerate}
\end{definition}

The following Theorem \ref{theorem-p-harmonic} can be proven in exactly the same way as its Riemannian counterpart found as Theorem 3.1 in \cite{Gud-Sob-1}.

\begin{theorem}\label{theorem-p-harmonic}
Let $\phi:(M,g)\to\cn$ be a complex-valued function on a semi-Riemannian manifold and $(\lambda,\mu)\in\cn^2\setminus\{0\}$ be such that the tension field $\tau$ and the conformality operator $\kappa$ satisfy 
$$\tau(\phi)=\lambda\cdot \phi\ \ \text{and}\ \ \kappa(\phi,\phi)=\mu\cdot\phi^2.$$
Then for any positive integer $p\in\zn^+$ the non-vanishing function 
$$\Phi_p:W=\{ x\in M \,|\, \phi(x) \not\in (-\infty,0] \}\to\cn$$ with 
$$\Phi_p(x)= 
\begin{cases}
c_1\cdot\log(\phi(x))^{p-1}, & \text{if }\; \mu = 0, \; \lambda \not= 0\\[0.2cm]	
c_1\cdot\log(\phi(x))^{2p-1}+ c_{2}\cdot\log(\phi(x))^{2p-2}, & \text{if }\; \mu \not= 0, \; \lambda = \mu\\[0.2cm]
c_{1}\cdot\phi(x)^{1-\frac\lambda{\mu}}\log(\phi(x))^{p-1} + c_{2}\cdot\log(\phi(x))^{p-1},	& \text{if }\; \mu \not= 0, \; \lambda \not= \mu
\end{cases}
$$ 
is a proper $p$-harmonic function.  Here $c_1,c_2$ are complex coefficients not both zero.
\end{theorem}


\section{Complex-Valued Harmonic Morphisms}
\label{section-harmonic-morphisms}


In this section we describe a method for constructing complex-valued harmonic morphisms $\phi:(M,g)\to\cn$ from semi-Riemannian manifolds. This is a special case of the much studied harmonic morphisms $\phi:(M,g)\to (N,h)$ between semi-Riemannian manifolds.  They are maps which pull back local real-valued harmonic functions on $(N,h)$ to harmonic functions on $(M,g)$.  The standard reference for the extensive theory of harmonic morphisms is the book \cite{Bai-Woo-book}, but we also recommend the updated online bibliography \cite{Gud-bib}. 
\medskip

The following result is a direct consequence of Theorem 3 of the paper  \cite{Fug-2} by B. Fuglede.

\begin{proposition}
A function $\phi:(M,g)\to\cn$ from a semi-Riemannian manifold to the standard Euclidean complex plane, is a {\it harmonic morphism} if and only if it is harmonic and horizontally conformal i.e. 
$$\tau(\phi)=0\ \ \text{and}\ \ \kappa(\phi,\phi)=0.$$
\end{proposition}

The following Theorem \ref{theorem-main-motivation} is a semi-Riemannian version of Theorem 5.2 of \cite{Bai-Eel}, see also \cite{Bai-Woo-book}.  It gives the theory of complex-valued harmonic morphisms a strong geometric flavour and provides a useful tool for the construction of minimal submanifolds of codimension two.  This is our main motivation for studying these maps.

\begin{theorem}\label{theorem-main-motivation}
Let $\phi:(M,g)\to\cn$ be a horizontally conformal map from a semi-Riemannian manifold to the standard Euclidean complex plane.  Then $\phi$ is harmonic if and only if its fibres are minimal at regular points of $\phi$.
\end{theorem}

The next result shows that eigenfamilies can be utilised to manufacture a variety of harmonic morphisms.

\begin{theorem}\cite{Gud-Sak-1}\label{theorem-rational}
	Let $(M,g)$ be a semi-Riemannian manifold and 
	$$\E =\{\phi_k:M\to\cn\,|\,k=1,2,\dots,n\}$$ 
	be an eigenfamily of complex-valued functions on $M$. If $P,Q:\cn^n\to\cn$ are linearily independent homogeneous polynomials of the same positive degree then the quotient
	$$\frac{P(\phi_1,\dots ,\phi_n)}{Q(\phi_1,\dots ,\phi_n)}$$ 
	is a non-constant harmonic morphism on the open and dense subset
	$$\{p\in M\,|\, Q(\phi_1(p),\dots ,\phi_n(p))\neq 0\}.$$
\end{theorem}


\section{The Semi-Riemannian Lie Group $\GLC n$}
\label{section-GLC}


The complex general linear group $\GLC n$ of invertible $n\times n$ matrices is given by 
$$\GLC n=\{ z\in\cn^{n\times n}\,|\, \det z\neq 0\}.$$
Its Lie algebra $\glc n$ of left-invariant vector fields on $\GLC n$ can be identified with $\cn^{n\times n}$ i.e. the linear space of complex $n\times n$ matrices.  We equip $\GLC n$ with its natural semi-Riemann\-ian metric $g$ induced by the semi-Euclidean inner product $\glc n\times\glc n\to\rn$ on $\glc n$ satisfying
$$g(Z,W)\mapsto -\,\Re\trace (Z\cdot W).$$
For $\glc n$ we then have the orthogonal decomposition 
$$\glc n=\glcp n\oplus\glcm n,$$
where $$\glcp n=\u n=\{Z\in\glc n\,|\, Z+\bar Z^t=0\}$$
is the set of skew-Hermitian matrices and 
$$\glcm n=i\cdot\u n=\{Z\in\glc n\,|\, Z-\bar Z^t=0\}$$ 
the set of the Hermitian ones.  Here $\u n$ is the Lie algebra of the unitary group $\U n$ which is the maximal compact subgroup of $\GLC n$ satisfying 
$$\U n=\{z\in\GLC n\,|\, z\, \bar z^t=I_n\}.$$

For $1\le r,s\le n$, we shall by $E_{rs}\in\rn^{n\times n}$ denote the matrix given by
$$
(E_{rs})_{\alpha\beta}=\delta_{r \alpha}\delta_{s\beta}
$$
and for $r<s$ let $X_{rs},Y_{rs}$ be the symmetric and skew-symmetric matrices
$$
X_{rs}=\frac 1{\sqrt 2}(E_{rs}+E_{sr}),\ \ Y_{rs}=\frac 1{\sqrt 2}(E_{rs}-E_{sr}),
$$ 
respectively.  Further, let $D_r$ be the diagonal elements with $D_t=E_{tt}$.
By $\B^+$ we denote the orthonormal basis for the Lie algebra $\u n$ satisfying 
$$\B^+=\{Y_{rs},i\, X_{rs}\,|\, 1\le r<s\le n\}\cup\{i\,D_{t}\,|\, t=1,2,\dots,n\}$$
and $\B^-=i\,\B^+$. Then $\B=\B^+\cup \B^-$ is an orthonormal basis for $\glc n$ such that $g(Z,Z)=1$ if $Z\in\B^+$ and $g(Z,Z)=-1$ for $Z\in\B^-$. For later use we define the two standard matrices $J_{n}$ and $I_{pq}$ by 
$$
J_{n}=
\begin{bmatrix}
0 & I_n\\
-I_n & 0
\end{bmatrix}\ \ \text{and}\ \ 
I_{pq}=\begin{bmatrix}
-I_p & 0\\
0 & I_q
\end{bmatrix}.
$$

Let $G$ be a subgroup of $\GLC n$ with Lie algebra $\g$ inheriting the induced left-invariant semi-Riemannian metric from $g$.  Then employing the Koszul formula for the Levi-Civita connection $\nabla$
on $(G,g)$, we see that for all $Z,W\in\g$ we have
\begin{eqnarray*}
	g(\nab ZZ,W)&=&g([W,Z],Z)\\
	&=&-\Re\trace (WZ-ZW)Z\\
	&=&-\Re\trace W(ZZ-ZZ)\\
	&=&0.
\end{eqnarray*}
If $Z\in\g$ is a left-invariant vector field on $G$ and $\phi:U\to\cn$ is a local complex-valued function on $G$ then the first and second order derivatives satisfy
\begin{equation*}\label{equation-diff-Z}
Z(\phi)(p)=\frac {d}{ds}\bigl(\phi(p\cdot\exp(sZ))\bigr)\Big|_{s=0},
\end{equation*}
\begin{equation*}\label{equation-diff-ZZ}
Z^2(\phi)(p)=\frac {d^2}{ds^2}\bigr(\phi(p\cdot\exp(sZ))\bigr)\Big|_{s=0}.
\end{equation*}
This implies that the tension field $\tau$ and the conformality operator $\kappa$ on $G$ fulfill 
\begin{eqnarray*}\label{equation-tau}
\tau(\phi)
&=&\sum_{Z\in\B_\g}g(Z,Z)\cdot \bigl(Z^2(\phi)-\nab ZZ(\phi)\bigr)\\
&=&\sum_{Z\in\B_\g}g(Z,Z)\cdot Z^2(\phi)\\ \\
\kappa(\phi,\psi)&=&\sum_{Z\in\B_\g}g(Z,Z)\cdot Z(\phi)\cdot Z(\psi),
\end{eqnarray*}
where $\B_\g$ is any orthonormal basis for the Lie algebra $\g$.
\smallskip

The restriction of the semi-Riemannian metric $g$ on $\GLC n$ to its maximal compact subgroup $\U n$ is its standard Riemannian metric.  For this we have the following result, see Lemma 5.1 of \cite{Gud-Sak-1}.

\begin{lemma}\label{lemma-U}
Let $z_{j\alpha}:\U n\to\mathbb{C}$ be the complex-valued matrix elements of the standard representation of the unitary group $\U n$. Then the tension field $\tau$ and the conformality operator $\kappa$ on $\U n$ satisfy the following relations
\begin{eqnarray*}
\tau(z_{j\alpha})&=&-\,n\cdot z_{j\alpha},\\
\kappa(z_{j\alpha},z_{k\beta})&=&-\,z_{k\alpha}\cdot z_{j\beta}. 
\end{eqnarray*}
\end{lemma}

With this at hand we yield the following statement.

\begin{proposition}\label{proposition-GLC}
Let $z_{j\alpha}:\GLC n\to\mathbb{C}$ be the complex-valued matrix elements of the standard representation of the general linear group $\GLC n$. Then the tension field $\tau$ and the conformality operator $\kappa$ on $\GLC n$ ful\-fill the following relations
\begin{eqnarray*}
\tau(z_{j\alpha})&=&-2\,n\cdot z_{j\alpha},\\
\kappa(z_{j\alpha},z_{k\beta})&=&-\,2\cdot z_{k\alpha}\cdot z_{j\beta}. 
\end{eqnarray*}
\end{proposition}

\begin{proof}
This is an immediate consequence of Lemma \ref{lemma-U} and how the semi-Riemannian metric is defined on the complex Lie algebra $$\glc n=\u n\oplus i\cdot\u n.$$
\end{proof}

\begin{theorem}\label{theorem-eigenfamily-GLC}
Let $v$ be a non-zero element of $\mathbb{C}^n$. Then the complex  $n$-dimensional vector space
$$
\E_v=\{\phi_a:\GLC n\to\mathbb{C}\,|\,\phi_a(z)=\trace(v^taz^t),\ a\in\mathbb{C}^n\}
$$
is an eigenfamily on $\GLC n$ such that for all $\phi,\psi\in\E_v$ we have 
$$\tau(\phi)=-\,2\,n\cdot\phi\ \ \text{and}\ \  \kappa(\phi,\psi)=-\,2\cdot\phi\cdot\psi.$$	
\end{theorem}

\begin{proof}
This result can be proven in exactly the same way as Theorem 5.2 presented in  \cite{Gud-Sak-1}.
\end{proof}


\section{The Semi-Riemannian Lie Group $\GLR n$}
\label{section-GLR}


The real general linear group $\GLR n$ of invertible $n\times n$ matrices is given by 
$$\GLR n=\{ x\in\rn^{n\times n}\,|\, \det x\neq 0\}.$$
The Lie algebra $\glc n$ of the complex general linear group $\GLC n$ has a natural orthogonal decomposition 
$$\glc n=\glr n\oplus i\cdot\glr n,$$ 
where $\glr n$ is the Lie algebra of $\GLR n$ consisting of the real $n\times n$ matrices.

\begin{proposition}\label{proposition-GLR}
Let $x_{j\alpha}:\GLR n\to\rn$ be the real-valued matrix elements of the standard representation of the general linear group $\GLR n$. Then the tension field $\tau$ and the conformality operator $\kappa$ on $\GLR n$ satisfy the following relations
\begin{eqnarray*}
\tau(x_{j\alpha})&=&-\,n\cdot x_{j\alpha},\\
\kappa(x_{j\alpha},x_{k\beta})&=&-\,x_{k\alpha}\cdot x_{j\beta}. 
\end{eqnarray*}
\end{proposition}

\begin{proof}
This is an immediate consequence of Proposition \ref{proposition-GLC} and how the semi-Riemannian metric is defined on the complex Lie algebra $$\glc n=\glr n\oplus i\cdot\glr n.$$
\end{proof}

As an immediate consequence of Proposition \ref{proposition-GLR} we have the following result.

\begin{theorem}\label{theorem-eigenfamily-GLR}
Let $v$ be a non-zero element of $\cn^n$. Then the complex  $n$-dimensional vector space
$$
\E_v=\{\phi_a:\GLR n\to\cn\,|\,\phi_a(x)=\trace(v^tax^t),\  a\in\mathbb{C}^n\}
$$
is an eigenfamily on $\GLR n$ such that for all $\phi,\psi\in\E_v$ we have 
$$\tau(\phi)=-\,n\cdot\phi\ \ \text{and}\ \  \kappa(\phi,\psi)=-\,\phi\cdot\psi.$$	
\end{theorem}

\begin{proof}
This is a direct consequence of Proposition \ref{proposition-GLR}  and Theorem \ref{theorem-eigenfamily-GLC}.
\end{proof}


\section{The Semi-Riemannian Lie Group $\GLH n$}
\label{section-GLH}


In this section we consider the quaternionic general linear group $\GLH n$.  Its standard complex representation $\pi:\GLH n\to\GLC {2n}$ is given by
$$
\pi:(z+jw)\mapsto g=\begin{bmatrix}
	z_{11}    &  \dots & z_{1n}   & w_{11}    & \dots  & w_{1n}\\
	\vdots    & \ddots & \vdots      & \vdots    & \ddots & \vdots  \\
	z_{n1} & \dots  & z_{nn} & w_{n1} & \dots  & w_{nn}\\
	-\bar w_{11} &  \dots   & -\bar w_{1n}   & \bar z_{11}  & \dots  & \bar z_{1n}\\
	\vdots       & \ddots   & \vdots      & \vdots    & \ddots & \vdots  \\
	-\bar w_{n1} & \dots & -\bar w_{nn} & \bar z_{n1} & \dots  & \bar z_{nn}\\
\end{bmatrix}.
$$
The Lie algebra $\glh{n}$ of $\GLH{n}$ clearly satisfies 
$$\glh{n}=\Big\{\begin{bmatrix}Z & W \\
	-\bar W & \bar Z\end{bmatrix}\Big|\ Z,W \in \glc {n}\Big\}.$$ 
As a subgroup of $\GLC{2n}$ the quaternionic general linear group $\GLH n$ inherits its natural semi-Riemannian metric $g$ induced by the
semi-Euclidean inner product $\glh n\times \glh n\to\glh n$ on $\glh n$ given by 
$$g(Z,W)=-\Re\trace(Z\cdot W).$$  
For the Lie algebra $\glh n$ we have the orthogonal splitting
$$
\glh n= \glhp n\oplus \glhm n,
$$
where
$$
\glhp n=\sp n=\Big\{\begin{bmatrix}
	Z & W\\
	-\bar{W} & \bar{Z}
\end{bmatrix}\in\cn^{2n\times 2n}\, \Big|\, Z+\bar Z^t=0, W-W^t=0\Big\}.
$$
By $\B^+$ we denote the following orthonormal basis for the Lie algebra $\sp n$ of the quaternionic unitary group $\Sp n$ which is the maximal compact subgroup of $\GLH n$. This satisfies
\begin{eqnarray*}
\B^+&=&
\Big\{\frac{1}{\sqrt{2}}
\begin{bmatrix}
		0 & iX_{rs}\\
		iX_{rs}& 0
	\end{bmatrix}, 
\frac{1}{\sqrt{2}}
\begin{bmatrix}
		0 & X_{rs}\\
		-X_{rs}& 0
\end{bmatrix}, 
\frac{1}{\sqrt{2}}
\begin{bmatrix}
		iX_{rs}& 0\\
		0 & -iX_{rs}
\end{bmatrix},\\
& &
\qquad
\frac{ 1}{\sqrt{2}}
\begin{bmatrix}
	Y_{rs}& 0\\
	0 & Y_{rs}
\end{bmatrix},
\frac{1}{\sqrt{2}}
\begin{bmatrix}
		0 & D_t\\
		-D_t & 0
\end{bmatrix}, 
\frac{1}{\sqrt{2}}
\begin{bmatrix}
		i D_t & 0\\
		0 & -i D_t
\end{bmatrix},\\
& &
\qquad\qquad
\frac{1}{\sqrt{2}}
\begin{bmatrix}
		0 & iD_t\\
		i D_t & 0
\end{bmatrix}
\Big|\, 1\leq r<s\leq n,\ 1 \leq t\leq n\Big\}\,.
\end{eqnarray*}
For the orthogonal complement $\glhm n$ of $\sp n$ in $\glh n$ we have the orthonormal basis
\begin{eqnarray*}
	\B^-&=&\Big\{\frac{1}{\sqrt{2}}\begin{bmatrix}
		0 & Y_{rs}\\
		-Y_{rs} & 0
	\end{bmatrix}, \frac{1}{\sqrt{2}}\begin{bmatrix}
		0 & iY_{rs}\\
		iY_{rs} & 0
	\end{bmatrix},\frac{1}{\sqrt{2}}\begin{bmatrix}
		iY_{rs} & 0\\
		0 & -i Y_{rs}
	\end{bmatrix},\\
& &
\qquad
\frac{1}{\sqrt{2}}\begin{bmatrix}
		X_{rs} & 0\\
		0 & X_{rs}
	\end{bmatrix}, \frac{1}{\sqrt{2}}\begin{bmatrix}
		D_t & 0\\
		0 & D_t
	\end{bmatrix} \Big|\, 1\leq r,s\leq n,\ 1 \leq t\leq n\Big\}\,.
\end{eqnarray*}
Then $\B=\B^+\cup\B^-$ is an orthonormal basis for $\glh n$ such that $g(Z,Z)=1$ if $Z\in\B^+$ and $g(Z,Z)=-1$ if $Z\in \B^-.$
\medskip

With this at our disposal we can now prove the following statement.

\begin{proposition}\label{proposition-GLH}
Let $z_{j\alpha}, w_{j\alpha}:\GLH n\to\cn$ be the complex-valued matrix elements of the standard representation of the quaternionic general linear group $\GLH n$. Then the tension field $\tau$ and the conformality operator $\kappa$ on $\GLH n$ satisfy the following relations
$$\tau(z_{j\alpha})= -\,2n\cdot z_{j\alpha},\ \ \tau(w_{j\alpha})=
-\,2n\cdot w_{j\alpha},$$
$$\kappa(z_{j\alpha},z_{k\beta})=-\,z_{k\alpha}\cdot z_{j\beta},\ \
\kappa(w_{j\alpha},w_{k\beta})=-\,w_{k\alpha}\cdot w_{j\beta},$$
$$\kappa(z_{j\alpha},w_{k\beta})=-\,z_{k\alpha}\cdot w_{j\beta}.$$
\end{proposition}

\begin{proof}
For the tension field $\tau$ on $\GLH n$ we have
	\begin{eqnarray*}
\tau(z_{j\alpha})
&=&\sum_{Z\in\B}g(Z,Z)\cdot Z^2(z_{j\alpha})\\
&=&\sum_{Z\in\B^+}Z^2(z_{j\alpha})-\sum_{Z\in\B^-}Z^2(z_{j\alpha})\\
&=&\frac{1}{2}\,e_jz\Big\{3\sum_{r<s}
\begin{bmatrix}
-X_{rs}^2&0\\
0&-X_{rs}^2
\end{bmatrix}
+\sum_{r<s}
\begin{bmatrix}
Y_{rs}^2 & 0\\
0& Y_{rs}^2 
\end{bmatrix}\\ 
&&\qquad\qquad\qquad\qquad\qquad\quad
+3\sum_{t=1}^n\begin{bmatrix}
			-D_t^2& 0\\
			0 & -D_t^2 
		\end{bmatrix}\Big\}e_\alpha^t\\
& & -\frac{1}{2}\,e_j z\Big\{\sum_{r<s}\begin{bmatrix}
			X_{rs}^2&0\\
			0& X_{rs}^2
		\end{bmatrix}+3\sum_{r<s}\begin{bmatrix}
			-Y_{rs}^2 & 0\\
			0 & -Y_{rs}^2
		\end{bmatrix}\\ 
&&\qquad\qquad\qquad\qquad\qquad\qquad\quad
	 +\sum_{t=1}^n\begin{bmatrix}
			D_t^2 & 0\\
			0 & D_t^2
		\end{bmatrix}\Big\}e_{\alpha}^t\\
&=& \frac{1}{2}\,e_jz\Big\{-4\sum_{r<s}\begin{bmatrix}
			X_{rs}^2 & 0\\
			0 & X_{rs}^2
		\end{bmatrix}+4\sum_{r<s}\begin{bmatrix}
			Y_{rs}^2 & 0\\
			0 & Y_{rs}^2
		\end{bmatrix}\\ 
& &\qquad\qquad\qquad\qquad\qquad\qquad\quad
-4\sum_{t=1}^n\begin{bmatrix}
			D_t^2 & 0\\
			0 & D_t^2
		\end{bmatrix}\Big\}e_{\alpha}^t\\
&=&-2\,e_jz\begin{bmatrix}
			n\cdot I_n &0\\
			0& n\cdot I_n
		\end{bmatrix}e_\alpha^t\\
&=& -2n\cdot z_{j\alpha}.
	\end{eqnarray*}

For the conformal operation $\kappa$ on $\GLH n$ we similarly we yield 
\begin{eqnarray*}
\kappa(z_{j\alpha},z_{k\beta})
&=&\sum_{Z\in\B}g(Z,Z)\cdot Z(z_{j\alpha})\cdot Z(z_{k\beta})\\
&=&\sum_{Z\in\B^+}Z(z_{j\alpha})\cdot Z(z_{k\beta})-\sum_{Z\in\B^-}Z(z_{j\alpha})\cdot Z(z_{k\beta})\\
&=&e_jz\Big\{\sum_{Z\in\B^+}Z\begin{bmatrix}
E_{\alpha\beta} & 0\\
	0 & 0
\end{bmatrix}Z^t-\sum_{Z\in\B^+}Z\begin{bmatrix}
			E_{\alpha\beta} & 0\\
			0 & 0
		\end{bmatrix}Z^t\Big\} z^te_k^t\\
		&=&e_jz\Big\{-\sum_{r<s}\begin{bmatrix}
			X_{rs}E_{\alpha\beta}X_{rs}& 0\\
			0 &0 
		\end{bmatrix}+\sum_{r<s}\begin{bmatrix}
			Y_{rs}E_{\alpha\beta}Y_{rs}^t & 0\\
			0 & 0
		\end{bmatrix} \\ 
	&& \qquad\qquad\qquad\qquad\qquad\qquad -\sum_{t=1}^n\begin{bmatrix}
			D_tE_{\alpha\beta}D_t & 0\\
			0 & 0
		\end{bmatrix}\Big\} z^te_k^t\\
		&=&-e_jz(E_{\beta\alpha})z^te_k^t\\
		&=&-z_{k\alpha}\cdot z_{j\beta}.
	\end{eqnarray*}
The other identities can be proven in exactly the same way.
\end{proof}

\begin{theorem}\label{theorem-eigenfamily-GLH}
Let $u,v$ be a non-zero elements of $\cn^n$. Then the complex $2n$-dimensional vector space 
$$\E_{uv}=\{\phi_{ab}:\GLH n\to\cn\ |\ \phi_{ab}(g)=\trace (u^taz^t+v^tbw^t),\  a,b\in \cn^n\}$$ 
is an eigenfamily on $\GLH n$ such that for all $\phi,\psi\in\E_{uv}$ we have 
$$\tau(\phi)=-\,2n\cdot\phi\ \ \text{and}\ \  \kappa(\phi,\psi)=-\,\phi\cdot\psi.$$
\end{theorem}


\section{The Semi-Riemannian Lie group $\SLC n$}
\label{section-SLC}


In this section we construct eigenfamilies on the semisimple  non-compact complex special linear group $\SLC n=\{z\in\GLC n\, |\,\det z=1\}$ 
equipped with its semi-Riemannian metric inherited from $\GLC{n}$. For the Lie algebra 
$$\slc n=\{Z\in\glc n\,|\, \trace Z=0\}$$ 
we have the orthogonal decomposition $\slc n=\su n\oplus i\cdot\su n$. Here 
$$\su n=\{Z\in\u n\,|\, \trace Z=0\}$$
is the Lie algebra of the special unitary group 
$$\SU n=\{z\in \U n\,|\, \det z=1\}$$
which is the maximal compact subgroup of $\SLC n$.

\begin{lemma}\label{lemma-SU}
Let $z_{j\alpha}:\SU n\to\mathbb{C}$ be the complex-valued matrix elements of the standard representation of the special unitary group $\SU n$. Then the tension field $\tau$ and the conformality operator $\kappa$ on $\SU n$ satisfy the following relations
\begin{eqnarray*}
\tau(z_{j\alpha})&=&-\,\frac{(n^2-1)}{n}\cdot z_{j\alpha},\\
\kappa(z_{j\alpha},z_{k\beta})&=&-\,(z_{k\alpha}\cdot  z_{j\beta}-\frac{1}{n}\cdot z_{j\alpha}\cdot z_{k\beta}). 
\end{eqnarray*}
\end{lemma}

\begin{proof}
For the Lie algebra $\u n$ of the unitary group $\U n$ we have the orthogonal splitting 
$$\u n=\su n\oplus\l,$$ 
where $\l$ is the real line generated by the unit vector $E_n=i\, I_n/\sqrt n$.
Hence the tension field $\hat\tau$ on the unitary group $\U n$ satisfies
\begin{eqnarray*}
\hat\tau(\phi)
&=&\tau(\phi)+ E_n^2(\phi),
\end{eqnarray*}
so we have
\begin{equation*}
\tau(z_{j\alpha})=\hat\tau(z_{j\alpha})- E_n^2(z_{j\alpha})=-(n\cdot z_{j\alpha}-\frac 1n\cdot z_{j\alpha})=-\,\frac{(n^2-1)}{n}\cdot z_{j\alpha}.
\end{equation*}

For the conformality operator $\hat\kappa$ on $\U n$ we similarily yield 
\begin{eqnarray*}
\hat\kappa(\phi,\psi)&=&\kappa(\phi,\psi)+E_n(\phi)\cdot E_n(\psi).
\end{eqnarray*}
Hence
\begin{eqnarray*}
\kappa(z_{j\alpha},z_{k\beta})
&=&\hat\kappa(z_{j\alpha},z_{k\beta})-E_n(z_{j\alpha})\cdot E_n(z_{k\beta})\\
&=&-\,( z_{k\alpha}\cdot z_{j\beta}-\frac{1}{n}\cdot z_{j\alpha}\cdot  z_{k\beta}).
\end{eqnarray*}
\end{proof}

For the special linear group $\SLC n$ we have the following statement.

\begin{proposition}\label{proposition-SLC}
Let $z_{j\alpha}:\SLC n\to\mathbb{C}$ be the complex-valued matrix elements of the standard representation of the special linear group $\SLC n$. Then the tension field $\tau$ and the conformality operator $\kappa$ on $\SLC n$ satisfy the following relations
\begin{eqnarray*}
\tau(z_{j\alpha})&=&-\,\frac{2\,(n^2-1)}{n}\cdot z_{j\alpha},\\
\kappa(z_{j\alpha},z_{k\beta})&=&-\,2\,( z_{k\alpha}\cdot z_{j\beta} -\frac{1}{n}\cdot z_{j\alpha}\cdot z_{k\beta}). 
\end{eqnarray*}
\end{proposition}

\begin{proof}
This is an immediate consequence of Lemma \ref{lemma-SU} and how the semi-Riemannian metric is defined on the complex Lie algebra $$\slc n=\su n\oplus i\cdot\su n.$$
\end{proof}

Let $P,Q:\SLC n\to\cn$ be homogeneous polynomials of the matrix elemens $z_{j\alpha}:\SLC n\to\cn$ of degree one i.e. of the form
$$
P(z)=\trace(A\cdot z^t)=\sum_{j,\alpha=1}^n a_{j\alpha}z_{j\alpha},\ \ Q(z)=\trace(B\cdot z^t)=\sum_{k,\beta}^nb_{k\beta} z_{k\beta}
$$
for some $A,B\in\cn^{n\times n}$. As a direct consequence of Proposition \ref{proposition-SLC} we see that
$$\tau(P)=-\,\frac{2\,(n^2-1)}{n}\cdot P,\ \ \tau(Q)=-\,\frac{2\,(n^2-1)}{n}\cdot Q$$
and
\begin{eqnarray*}
&&\kappa(P,Q)+\frac{2\,(n-1)}{n}\cdot P\,Q\\
&=&\sum_{j,\alpha,k,\beta=1}^n a_{j\alpha}b_{k\beta}\,\kappa(z_{j\alpha},z_{k\beta})+\frac{2\,(n-1)}{n}\sum_{j,\alpha,k,\beta=1}^n a_{j\alpha}b_{k\beta}\,z_{j\alpha}z_{k\beta}\\
&=&-\,2\sum_{j,\alpha,k,\beta=1}^n a_{j\alpha} b_{k\beta}\, z_{j\beta}z_{k\alpha}+\frac{2}{n}\sum_{j,\alpha,k,\beta=1}^n a_{j\alpha}b_{k\beta}\,z_{j\alpha} z_{k\beta}\\
& &+\frac{2\,(n-1)}{n}\sum_{j,\alpha,k,\beta=1}^n a_{j\alpha}b_{k\beta}\, z_{j\alpha}z_{k\beta}\\
&=&2\sum_{j,\alpha,k,\beta=1}^n\left(a_{j\alpha}b_{k\beta}\,z_{j\alpha}z_{k\beta}-a_{j\alpha}b_{k\beta}\,z_{j\beta}z_{k\alpha}\right)\\
&=&2\sum_{j,\alpha,k,\beta=1}^n(a_{j\alpha}b_{k\beta}-a_{k\alpha}b_{j\beta})\,z_{j\alpha}z_{k\beta}.
\end{eqnarray*}
\smallskip

\begin{theorem}\label{theorem-eigenfamily-SLC}
Let $v$ be a non-zero element of $\mathbb{C}^n$. Then the complex  $n$-dimensional vector space
$$
\E_v=\{\phi_a:\SLC n\to\mathbb{C}\,|\,\phi_a(z)=\trace(v^taz^t),\  a\in\mathbb{C}^n\}
$$
is an eigenfamily on $\SLC n$ such that for all $\phi,\psi\in\E_v$ we have 
$$\tau(\phi)=-\,\frac{2\,(n^2-1)}{n}\cdot\phi,\ \ \kappa(\phi,\psi)=-\frac{2\,(n-1)}{n}\cdot\phi\cdot \psi.$$
\end{theorem}

\begin{proof}
Assume that $a,b\in\mathbb{C}^n$ and define $A=v^ta$ and $B=v^tb\,.$ By construction any two columns of the matrices $A$ and $B$ are linearly dependent. This means that for all $1\leq j,\alpha,k,\beta\leq n$
$$
\det\begin{bmatrix}
a_{j\alpha} & b_{j\beta}\\
a_{k\alpha} & b_{k\beta}
\end{bmatrix}=a_{j\alpha}b_{k\beta}-a_{k\alpha}b_{j\beta}=0.
$$
The statement now follows from the calculation above.
\end{proof}


\section{The Semi-Riemannian Lie Group $\SLR n$}
\label{section-SLR}


In this section we construct eigenfamilies on the semisimple non-compact special linear group $\SLR n$ equipped with its semi-Riemannian metric inherited from $\GLC{n}$. 
The special linear group $\SLR n$ is the subgroup of $\GLR n$ satisfying 
$$\SLR n=\{x\in\GLR n\, |\,\det x=1\}$$
with Lie algebra $$\slr n=\{X\in\glr n\,|\, \trace X=0\}.$$

\begin{proposition}\label{proposition-SLR}
Let $x_{j\alpha}:\SLR n\to\mathbb{R}$ be the real-valued matrix elements of the standard representation of the special linear group $\SLR n$. Then the tension field $\tau$ and the conformality operator $\kappa$ on $\SLR n$ satisfy the following relations
\begin{eqnarray*}
\tau(x_{j\alpha})&=&-\,\frac{(n^2-1)}{n}\cdot x_{j\alpha},\\
\kappa(x_{j\alpha},x_{k\beta})&=&-\, (x_{j\beta}\cdot  x_{k\alpha}-\frac{1}{n}\cdot x_{j\alpha}\cdot  x_{k\beta}). 
\end{eqnarray*}
\end{proposition}

\begin{proof}
	The Lie algebra $\slc n$ of the complex special linear group $\SLC n$ is the complexification of $\slr n$ and we have the orthogonal decomposition
	$$\slc n=\slr n\oplus i\cdot\slr n.$$
	Hence the statement is an immediate consequence of Lemma \ref{proposition-SLC}.
\end{proof}

The next result is a direct consequence of Theorem \ref{theorem-eigenfamily-SLC}, Propositions  \ref{proposition-SLC} and \ref{proposition-SLR}.

\begin{theorem}\label{theorem-eigenfamily-SLR}
Let $v$ be a non-zero element of $\cn^n$. Then the complex $n$-dimensional vector space
$$\E_v = \{\phi_a: \SLR n \to \cn\, |\, \phi_a(x) = \trace(v^tax^t),\ a\in\cn^n\}$$
is an eigenfamily on $\SLR n$ such that for all $\phi, \psi \in \E_v$ we have
$$
\tau(\phi)= -\frac{(n^2-1)}{n}\cdot\phi,\quad \kappa(\phi, \psi) = -\frac{(n-1)}{n}\cdot\phi\cdot \psi\,.
$$
\end{theorem}

\renewcommand{\arraystretch}{2}
\begin{table}[h]
	\makebox[\textwidth][c]{
		\begin{tabular}{ccccc}
			\midrule\midrule
			Lie group 	& Eigenfunctions $\phi$  & $\lambda$ & $\mu$ & Conditions\\
			\midrule\midrule
			$\GLC{n}$	& $\trace(v^taz^t)$	& ${-\,2\, n}$ & $-\,2$ & $a\in\cn^n$ \\
			\midrule
			$\GLR{n}$	& $\trace(v^tax^t)$	& ${-\, n}$ & $-\,1$ & $a\in\cn^n$ \\
			\midrule
			$\GLH n$    & $\trace (u^taz^t+v^tbw^t)$ & $-\,2\,n$ & $-\,1$      & $a,b\in \cn^n$  \\
			\midrule
			$\SLC{n}$	& $\trace(v^taz^t)$	& $-\frac{2(n^2-1)}{n}$ & $-\frac{2(n-1)}{n}$ & $a\in\cn^n$ \\
			\midrule
			$\SLR n$    & $\trace(v^tax^t)$ & $-\frac{(n^2-1)}{n}$ & $-\frac{(n-1)}{n}$ & $a\in\cn^n$ \\
			\midrule
			$\SLH n$ & $\trace (u^taz^t+v^tbw^t)$ & $-\,\frac{(4n^2-1)}{2n}$ & $-\,\frac{(2n-1)}{2n}$ & $a,b\in \cn^n$ \\
			\midrule\midrule
		\end{tabular}
	}
	\bigskip
	\caption{Eigenfunctions on classical non-compact Lie groups.}
	\label{table-eigenfunctions-1}	
\end{table}
\renewcommand{\arraystretch}{1}


\section{The Semi-Riemannian Lie Group $\SLH n\cong\SUs{2n}$}
\label{section-SLH}


In this section we construct eigenfamilies on the semisimple non-compact quaternionic special linear group $\SLH n$.  This can be realised as 
$$\SUs{2n}=\Big\{\begin{bmatrix}z & w \\ -\bar w & \bar z\end{bmatrix}\in\GLC{2n}\,\Big|\,(z+j\,w)\in\SLH n\Big\},$$
with Lie algebra
$$
\sus {2n}=\Big\{\begin{bmatrix}
	Z & W\\
	-\bar{W} & \bar{Z}
\end{bmatrix}\in\glh n\, \Big|\, \Re\trace Z=0\Big\}.
$$
For the Lie algebra $\glh n$ of $\GLH n$ we have the orthogonal decomposition 
$$\glh n=\sus {2n} \oplus \l$$
where $\l$ is the real line in $\glh n$ generated by the unit vector $E_{2n}=I_{2n}/{\sqrt{2n}}$.

\begin{proposition}\label{proposition-SUs}
Let $z_{j\alpha}, w_{k\beta}:\SUs {2n}\to\cn$ be the complex-valued matrix elements of the standard representation of the Lie group $\SUs {2n}$. Then the tension field $\tau$ and the conformality operator $\kappa$ on $\SUs {2n}$ satisfy the following relations
$$
\tau(z_{j\alpha})=-\,\frac{(4n^2-1)}{2n}\cdot z_{j\alpha},\ \
\tau(w_{j\alpha})=-\,\frac{(4n^2-1)}{2n}\cdot w_{j\alpha},
$$
$$
\kappa(z_{j\alpha},z_{k\beta})=-\,(z_{k\alpha}\cdot z_{j\beta}-\frac{1}{2n}\cdot z_{j\alpha}\cdot z_{k\beta}),$$
$$
\kappa(z_{j\alpha},w_{k\beta})=-\,(z_{k\alpha}\cdot w_{j\beta}-\frac{1}{2n}\cdot z_{j\alpha}\cdot w_{k\beta}),
$$
$$
\kappa(w_{j\alpha},w_{k\beta})=-\,(w_{k\alpha}\cdot w_{j\beta}-\frac{1}{2n}\cdot w_{j\alpha}\cdot w_{k\beta}).
$$
\end{proposition}

\begin{proof}
Let $\hat\tau$ and $\hat\kappa$ denote the tension field and the conformality operator on $\GLH n$, respectively.  Then it follows from Proposition \ref{proposition-GLH} and the orthogonal decomposition $\glh n=\sus {2n} \oplus \l$ that
\begin{eqnarray*}
\tau(z_{j\alpha})&=&\hat{\tau}(z_{j\alpha})+E_n^2(z_{j\alpha})\\
&=& -2n\cdot  z_{j\alpha}+\frac{1}{2n}\cdot z_{j\alpha}\\
&=&-\,\frac{(4n^2-1)}{2n}\cdot z_{j\alpha}.
\end{eqnarray*}
Similarly, we have 
\begin{eqnarray*}
\kappa(z_{j\alpha},z_{k\beta})
&=&\hat{\kappa}(z_{j\alpha},z_{k\beta})
+E_n(z_{j\alpha})\cdot E_n(z_{k\beta})\\
&=&-\,(z_{k\alpha}\cdot z_{j\beta}
-\frac{1}{2n}\cdot z_{j\alpha}\cdot z_{k\beta}).
\end{eqnarray*}
\end{proof}

\begin{theorem}\label{theorem-eigenfamily-SUs}
Let $u,v$ be a non-zero elements of $\cn^n$. Then the complex $2n$-dimensional vector space 
$$\E_{uv}=\{\phi_{ab}:\SUs{2n}\to\cn\ |\ \phi_{ab}(g)=\trace (u^taz^t+v^tbw^t),\  a,b\in \cn^n\}$$ 
is an eigenfamily on $\SUs{2n}$ such that for all $\phi,\psi\in\E_{uv}$ we have 
$$\tau(\phi)=-\,\frac{(4n^2-1)}{2n}\cdot\phi\ \ 
\text{and}\ \  
\kappa(\phi,\psi)=-\,\frac{(2n-1)}{2n}\cdot\phi\cdot\psi.$$
\end{theorem}

\begin{proof}
Here the statement can be proven in the exactly the same way as that of Theorem \ref{theorem-eigenfamily-SLC}.
\end{proof}


\section{The Semi-Riemannian Lie Group $\SOC n$}
\label{section-SOC}


The semisimple complex special orthogonal group $\SOC n$ is the subgroup of $\GLC n$ defined by 
$$\SOC n=\{z\in\SLC n\,|\, z\cdot z^t=I_n\}.$$
Its Lie algebra 
$$\soC n=\{Z\in\glc n\,|\,Z+Z^t=0\}$$ 
has the orthogonal decomposition 
$$\soC n=\so n\oplus i\cdot\so n,$$
where $\so n$ is the Lie algebra of the special orthogonal group $\SO n$  consisting of the real skew-symmetric $n\times n$ matrices.  The restriction of the semi-Riemannian metric $g$ on $\SOC n$ to its maximal compact subgroup $\SO n$ is its standard Riemannian metric.  For this we have the following result, see Lemma 4.1 of \cite{Gud-Sak-1}.

\begin{lemma}\label{lemma-SO}
Let $x_{j\alpha}:\SO n\to\mathbb{R}$ be the real-valued matrix elements of the standard representation of the special orthogonal group $\SO n$. Then the tension field $\tau$ and the conformality operator $\kappa$ on $\SO n$ satisfy the following relations
\begin{eqnarray*}
\tau(x_{j\alpha})&=&-\,\frac{(n-1)}2\cdot x_{j\alpha},\\
\kappa(x_{j\alpha},x_{k\beta})
&=&-\,\frac 12\cdot ( x_{j\beta}\cdot  x_{k\alpha}
-\delta_{jk}\cdot \delta_{\alpha\beta}). 
\end{eqnarray*}
\end{lemma}

\begin{proposition}\label{proposition-SOC}
Let $z_{j\alpha}:\SOC n\to\mathbb{C}$ be the complex-valued matrix elements of the standard representation of the complex special orthogonal group $\SOC n$. Then the tension field $\tau$ and the conformality operator $\kappa$ on $\SOC n$ satisfy the following relations
\begin{eqnarray*}
\tau(z_{j\alpha})&=&-\,(n-1)\cdot z_{j\alpha},\\
\kappa(z_{j\alpha},z_{k\beta})
&=&-\,( z_{j\beta}\cdot z_{k\alpha}
-\delta_{jk}\cdot \delta_{\alpha\beta}). 
\end{eqnarray*}
\end{proposition}

\begin{proof}
This is an immediate consequence of Lemma \ref{lemma-SO} and the fact how the semi-Riemannian metric is defined on the complex Lie algebra $$\soC n=\so n\oplus i\cdot\so n.$$
\end{proof}

\begin{theorem}\label{theorem-eigenfamily-SOC}
Let $v\in\cn^n$ be a non-zero isotropic element i.e. $(v,v)=0$,
then the complex $n$-dimensional vector space
$$\E_v=\{\phi_a:\SOC n\to\cn\ |\ \phi_a(z)=\trace (v^taz^t),\
a\in \cn^n\}$$ is an eigenfamily on $\SOC n$ such that for all $\phi,\psi\in\E_v$ we have 
$$\tau(\phi)=-\,(n-1)\cdot\phi,\ \ 
\kappa(\phi,\psi)=-\,\phi\cdot \psi.$$
\end{theorem}

\begin{proof}
The Lie algebra $\soC n$ of the complex special linear group $\SOC n$ is the complexification of $\so n$ and we have the orthogonal decomposition
$$\soC n=\so n\oplus i\cdot\so n.$$
Hence the statement is an immediate consequence of Theorem 4.3 of  \cite{Gud-Sak-1}.
\end{proof}


\section{The Semi-Riemannian Lie Group $\SpC n$}
\label{section-SpC}


The semisimple complex symplectic group $\SpC n$ is the subgroup of $\SLC {2n}$ with
$$\SpC n=\{z\in\SLC {2n}\,|\, z\,J_{n}\,z^t=J_{n}\}$$
and Lie algebra 
$$\spC n=\{Z\in\slc {2n}\,|\, Z\,J_{n}+J_{n}\,Z^t=0\}.$$ 
The maximal compact subgroup of $\SpC n$ is the quaternionic unitary group 
$$\Sp n=\{z\in\SpC n\, | \ z\,\bar z^t=I_{2n}\}$$
with Lie algebra 
$$\sp n=\{Z\in\spC n\ |\, Z+\bar Z^t=0\}.$$
For $\spC n$ we have the orthogonal decomposition 
$$\spC n=\sp n\oplus i\cdot\sp n.$$
The restriction of the semi-Riemannian metric $g$ on $\SpC n$ to $\Sp n$ is its standard Riemannian metric.  For this we have the following result, see Lemma 6.1 of \cite{Gud-Sak-1} and Lemma 6.1 of \cite{Gud-Mon-Rat-1}.

\begin{lemma}\label{lemma-Sp}
Let $z_{j\alpha},w_{j\alpha}:\Sp n\to\mathbb{C}$ be the complex-valued matrix elements of the standard representation of the quaternionic unitary group $\Sp n$. Then the tension field $\tau$ and the conformality operator $\kappa$ on $\Sp n$ satisfy the following relations
$$
\tau(z_{j\alpha})= -\,\frac{2n+1}2\cdot z_{j\alpha},
\ \ \tau(w_{k \beta})=-\,\frac{2n+1}2\cdot w_{k \beta},
$$
$$
\kappa(z_{j\alpha},z_{k\beta})
=-\,\frac 12\cdot z_{k\alpha}\cdot z_{j\beta},\ \
\kappa(w_{j\alpha},w_{k\beta})
=-\,\frac 12\cdot w_{k\alpha}\cdot w_{j\beta},
$$
$$
\kappa(z_{j\alpha},w_{k\beta})
=-\,\frac 12\cdot z_{k\alpha}\cdot w_{j\beta}.
$$
\end{lemma}

With this at hand we then yield the following result.

\begin{proposition}\label{proposition-SpC}
Let $z_{j\alpha},w_{j\alpha}:\SpC n\to\mathbb{C}$ be the complex-valued matrix elements of the standard representation of the quaternionic unitary group $\Sp n$. Then the tension field $\tau$ and the conformality operator $\kappa$ on $\Sp n$ satisfy the following relations
$$
\tau(z_{j\alpha})= -\,(2n+1)\cdot z_{j\alpha},\ \ 
\tau(w_{k \beta})= -\,(2n+1)\cdot w_{k \beta},
$$
$$
\kappa(z_{j\alpha},z_{k\beta})=-\, z_{k\alpha}\cdot z_{j\beta},\ \
\kappa(w_{j\alpha},w_{k\beta})=-\, w_{k\alpha}\cdot w_{j\beta},
$$
$$
\kappa(z_{j\alpha},w_{k\beta})=-\, z_{k\alpha}\cdot w_{j\beta}.
$$
\end{proposition}

\begin{proof}
This is an immediate consequence of Lemma \ref{lemma-Sp} and  how the semi-Riemannian metric is defined on the complex Lie algebra 
$$\spC n=\sp n\oplus i\cdot\sp n.$$
\end{proof}

\begin{theorem}\label{theorem-eigenfamily-SpC}
Let $u,v\in\cn^n$ be a non-zero elements of $\cn^n$,
then the complex $2n$-dimensional vector space
$$\E_{uv}=\{\phi_{ab}:\SpC n\to\cn\ |\ \phi_{ab}(g)=\trace (u^taz^t+v^tbw^t),\ a,b\in \cn^n\}$$ 
is an eigenfamily on $\SpC n$ such that for all $\phi,\psi\in\E_{uv}$ we have 
$$\tau(\phi)=-\,(2n+1)\cdot\phi,\ \ \kappa(\phi,\psi)=-\,\phi\cdot\psi.$$
\end{theorem}

\begin{proof}The statement is an immediate consequence of Proposition \ref{proposition-SpC}.
\end{proof}


\section{The Semi-Riemannian Lie Group $\SpR n$}
\label{section-SpR}


The semisimple real symplectic group $\SpR n$ is the subgroup of the complex symplectic group $\SpC{n}$ given by 
$$\SpR n=\{x\in\SLR {2n}\,|\, x\, J_{n}\,x^t=J_{n}\}$$
with Lie algebra 
$$\spR n=\{X\in\glr{2n}\,|\, X\, J_{n}+J_{n}\, X^t=0\}.$$
For the Lie algebra $\spC n$ we have the orthogonal decomposition 
$$\spC n=\spR n\oplus i\,\spR n.$$

\begin{theorem}\label{theorem-eigenfamily-SpR}
Let $v$ be a non-zero element of $\cn^n$,
then the complex $n$-dimensional vector space
$$\E_v=\{\phi_{a}:\SpR n\to\cn\ |\ \phi_{a}(x)=\trace (v^tax^t),\ a\in \cn^n\}$$ 
is an eigenfamily on $\SpR n$ such that for all $\phi,\psi\in\E_v$ we have 
$$\tau(\phi)=-\,\frac{2n+1}2\cdot\phi,\ \ \kappa(\phi,\psi)=-\,\frac 12\cdot\phi\cdot\psi.$$
\end{theorem}

\begin{proof}
The result follows directly from Theorem \ref{theorem-eigenfamily-SpC} and the fact that $\spC n=\spR n\oplus i\,\spR n$.
\end{proof}


\section{The Semi-Riemannian Lie Group $\SOs {2n}$}
\label{section-SOs}


In this section we construct eigenfamilies of complex-valued 
functions on the semisimple non-compact Lie group
$$\SOs{2n}=\{g\in\SUU nn\ |\ g\cdot I_{nn}\cdot J_{n}\cdot
g^t=I_{nn}\cdot J_{n}\},$$ 
where 
$$\SUU nn =\{z\in\SLC{2n}\ |\ z\cdot I_{n,n}\cdot z^*=I_{n,n}\}.$$
For the Lie algebra 
$$\sos{2n}=\Big\{\begin{bmatrix}Z & W \\
	-\bar{W} & \bar{Z}\end{bmatrix}\in\cn^{2n\times 2n}\Big|\ Z+Z^*=0\
\text{and}\ W+W^t=0\Big\}$$
of $\SOs{2n}$ we have the orthogonal splitting
${\mathfrak{so}_+^*}(2n)\oplus{\mathfrak{so}_-^*}(2n)$,  
where the subspaces ${\mathfrak{so}_+^*}(2n)$ and ${\mathfrak{so}_-^*}(2n)$ have the orthonormal basis $\B^+$ and $\B^-$, respectively, with
\begin{eqnarray*}
\B^+&=&\Big\{\frac 1{\sqrt 2}
\begin{bmatrix}
Y_{rs} & 0 \\
0 & Y_{rs}
\end{bmatrix},
\frac 1{\sqrt 2}
\begin{bmatrix}
iX_{rs} & 0 \\
0 & -iX_{rs}
\end{bmatrix},
\frac 1{\sqrt 2}
\begin{bmatrix}
iD_t & 0 \\
0 & -iD_t
\end{bmatrix}\\
& &\qquad\qquad\qquad\qquad\qquad\qquad
\Big|\ 1\le r<s\le n\ \ \text{and}\ \ 1\le t\le n\Big\}
\end{eqnarray*}
and
$$
\B^-=
\left\{\frac 1{\sqrt 2}
\begin{bmatrix}
0 & Y_{rs} \\
-Y_{rs} & 0
\end{bmatrix},
\frac 1{\sqrt 2}
\begin{bmatrix}
0 & iY_{rs} \\
iY_{rs} & 0
\end{bmatrix}\Big|\ 1\le r<s\le n \right\}.
$$
Here $g(Z,Z)=1$ for all $Z\in\B^+$ and $g(Z,Z)=-1$ for all $Z\in \B^-$.

\begin{proposition}\label{proposition-SOs}
Let $z_{j\alpha}, w_{k\beta}:\SOs{2n}\to\mathbb{C}$ be the complex-valued matrix coefficients of the standard representation of $\SOs{2n}$. Then the tension field $\tau$ and the conformality operator $\kappa$ on $\SOs{2n}$ satisfy the following relations
$$\tau(z_{j\alpha})=-\,\frac{2n-1}{2}\cdot z_{j\alpha},\ \ \tau(w_{j\alpha})=-\,\frac{2n-1}{2}\cdot w_{j\alpha},$$
$$\kappa(z_{j\alpha}, z_{k\beta})= -\,\frac 12\cdot z_{k\alpha}\cdot z_{j\beta}, \ \ \kappa(w_{j\alpha},w_{k\beta})= -\,\frac 12\cdot w_{k\alpha}\cdot w_{j\beta},
$$
$$
{\kappa}(z_{j\alpha},w_{k\beta})
=-\frac{1}{2}\cdot z_{k\alpha}\cdot w_{j\beta}.
$$
\end{proposition}

\begin{proof}
Here we can apply exactly the same strategy as for the proof of Proposition \ref{proposition-GLH}.
\end{proof}

\begin{theorem}\label{theorem-eigenfamily-SOs}
Let $u,v\in\cn^n$ be a non-zero elements of $\cn^n$,
then the complex $2n$-dimensional vector space
$$\E_{uv}=\{\phi_{ab}:\SOs{2n}\to\cn\ |\ \phi_{ab}(g)=\trace (u^taz^t+v^tbw^t),\ a,b\in \cn^n\},$$ 
is an eigenfamily on $\SOs{2n}$ such that for all $\phi,\psi\in\E_{uv}$ we have 
$$\tau(\phi)=-\,\frac{2n-1}{2}\cdot\phi,\ \ \kappa(\phi,\psi)=-\,\frac 12\cdot\phi\cdot\psi.$$
\end{theorem}

\begin{proof}The statement is an immediate consequence of Proposition \ref{proposition-SOs}.
\end{proof}


\section{The Semi-Riemannian Lie Group $\SUU pq$}
\label{section-SUU}


In this section we construct eigenfamilies on the non-compact semisimple Lie group 
$$\SUU pq=\{z\in\SLC{p+q}\ |\ z\cdot I_{p,q}\cdot z^*=I_{p,q}\}.$$
For its Lie algebra
$$\suu pq=\{Z\in\slc{p+q}\ |\ Z\cdot I_{p,q}+ I_{p,q}\cdot Z^*=0\}$$
we have a natural orthogonal splitting 
$$\suu pq=\mathfrak{s}(\u p+\u q)\oplus i\cdot\m$$
such that 
$$\su{p+q}=\mathfrak{s}(\u p+\u q)\oplus\m$$ is the Lie algebra of the special orthogonal group $\SU{p+q}$.

\begin{proposition}\label{proposition-SUU}
Let $z_{j\alpha}:\SUU pq\to\mathbb{C}$ be the complex-valued matrix elements of the standard representation of the special unitary group $\SUU pq$. Then the tension field $\tau$ and the conformality operator $\kappa$ on $\SUU pq$ satisfy the following relations
\begin{eqnarray*}
\tau(z_{j\alpha})&=&-\,\frac{(p+q)^2-1}{(p+q)}\cdot z_{j\alpha},\\
\kappa(z_{j\alpha},z_{k\beta})&=&-\,(z_{k\alpha}\cdot  z_{j\beta}-\frac{1}{(p+q)}\cdot z_{j\alpha}\cdot z_{k\beta}). 
	\end{eqnarray*}
\end{proposition}

\begin{proof}
This is an immediate consequence of Lemma \ref{lemma-SU}, the above relationship between the Lie algebras $\su{p+q}$ and $\suu pq$ and how the semi-Riemannian metric $g$ is defined on the complex Lie algebra $\glc {p+q}$.
\end{proof}

\begin{theorem}\label{theorem-eigenfamily-SUUpq}
Let $v$ be a non-zero element of $\mathbb{C}^{p+q}$. Then the complex  $(p+q)$-dimensional vector space
$$
\E_v=\{\phi_a:\SUU pq\to\mathbb{C}\,| \,\phi_a(z)=\trace(v^taz^t),\  a\in\mathbb{C}^{p+q}\}
$$
is an eigenfamily on $\SUU pq$ such that for all $\phi,\psi\in\E_v$ we have 
$$\tau(\phi)=-\,\frac{(p+q)^2-1}{(p+q)}\cdot\phi,\ \ \kappa(\phi,\psi)=-\frac{(p+q-1)}{(p+q)}\cdot\phi\cdot \psi.$$
\end{theorem}

\begin{proof}
The statement follows directly from Proposition \ref{proposition-SUU}.
\end{proof}


\section{The Semi-Riemannian Lie Group $\SOO pq$}
\label{section-SOO}


In this section we construct eigenfamilies on the non-compact semisimple Lie group 
$$\SOO pq=\{x\in\SLR{p+q}\ |\ x\cdot I_{p,q}\cdot x^t=I_{p,q}\}.$$
For its Lie algebra
$$\soo pq=\{X\in\slr{p+q}\ |\ X\cdot I_{p,q}+ I_{p,q}\cdot X^t=0\}$$
we have a natural orthogonal spilling
$$\soo pq\cong (\so p\oplus\so q)\oplus i\cdot\m$$
such that 
$$\so{p+q}=(\so p\oplus\so q)\oplus\m$$ is the Lie algebra of the special orthogonal group $\SO{p+q}$.

\begin{proposition}\label{proposition-SOO}
Let $x_{j\alpha}:\SOO pq\to\mathbb{R}$ be the real-valued matrix elements of the standard representation of the special orthogonal group $\SOO pq$. Then the tension field $\tau$ and the conformality operator $\kappa$ on $\SOO pq$ satisfy the following relations
\begin{eqnarray*}
\tau(x_{j\alpha})&=&-\,\frac{(p+q-1)}2\cdot x_{j\alpha},\\
\kappa(x_{j\alpha},x_{k\beta})
&=&-\,\frac 12\cdot ( x_{j\beta}\cdot  x_{k\alpha}
-\delta_{jk}\cdot \delta_{\alpha\beta}). 
\end{eqnarray*}
\end{proposition}

\begin{proof}
This is an immediate consequence of Lemma \ref{lemma-SO}, the above relationship between the Lie algebras $\so{p+q}$ and $\soo pq$ and how the semi-Riemannian metric $g$ is defined on the complex Lie algebra $\glr {p+q}$.
\end{proof}

\begin{theorem}\label{theorem-eigenfamily-SOO}
Let $v\in\cn^{p+q}$ be a non-zero isotropic element i.e. $(v,v)=0$,
then the complex $(p+q)$-dimensional vector space
$$\E_v=\{\phi_a:\SOO pq\to\cn\ |\ \phi_a(z)=\trace (v^taz^t),\ a\in \cn^{p+q}\}$$ 
is an eigenfamily on $\SOO pq$ such that for all $\phi,\psi\in\E_v$ we have 
$$\tau(\phi)=-\,\frac{(p+q-1)}2\cdot\phi,\ \ 
	\kappa(\phi,\psi)=-\,\frac 12\cdot\phi\cdot\psi.$$
\end{theorem}

\begin{proof}
The statement follows directly from Proposition \ref{proposition-SOO}.
\end{proof}

\renewcommand{\arraystretch}{2}
\begin{table}[h]
	\makebox[\textwidth][c]{
		\begin{tabular}{ccccc}
			\midrule
			\midrule
			Lie group 	
			& Eigenfunctions $\phi$  & $\lambda$ & $\mu$ & Conditions \\
			\midrule
			\midrule
			$\SOC n$	
			& $\trace(v^taz^t)$	& $-\,(n-1)$ & $-1$& $a\in\cn^n,\ (v,v)=0$ \\
			\midrule
			$\SpC n$    
			& $\trace (u^taz^t+v^tbw^t)$ & $-\,(2n+1)$ & $-1$ & $a,b\in\cn^n$ \\
			\midrule
			$\SpR n$	
			& $\trace(v^tax^t)$ & $-\,\frac{2n+1}2$& $-\frac 12$ & $a\in\cn^n$ \\
			\midrule
			$\SOs{2n}$	
			& $\trace (u^taz^t+v^tbw^t)$ & $-\frac{2n-1}{2}$ & $-\,\frac 12$ & $a,b\in\cn^n$ \\
			\midrule
			$\SUU pq$	
			& $\trace(v^taz^t)$ & $-\frac{(p+q)^2-1}{(p+q)}$ & $-\frac{(p+q-1)}{(p+q)}$ & $a\in\cn^n$ \\
			\midrule
			$\SOO pq$	
			& $\trace(v^tax^t)$ & $-\frac{(p+q-1)}2$ & $-\frac 12$ & $a\in\cn^n$ \\
			\midrule
			$\Spp pq$ 
			& $\trace (u^taz^t+v^tbw^t)$ & $-\frac{2(p+q)+1}2$ & $-\frac 12$ & $a,b\in\cn^n$ \\
			\midrule\midrule
		\end{tabular}
	}
	\bigskip
	\caption{Eigenfunctions on classical non-compact Lie groups.}
	\label{table-eigenfunctions-2}	
\end{table}
\renewcommand{\arraystretch}{1}


\section{The Semi-Riemannian Lie Group $\Spp pq$}
\label{section-Spp}


In this section we construct eigenfamilies on the non-compact semisimple Lie group 
$$\Spp pq=\{g\in\SLH{p+q}\ |\ g\cdot I_{p,q}\cdot g^*=I_{p,q}\}.$$
For its Lie algebra
$$\spp pq=\{Z\in\slh{p+q}\ |\ Z\cdot I_{p,q}+ I_{p,q}\cdot Z^*=0\}$$
we have a natural orthogonal decomposition
$$\spp pq=(\sp p\oplus\sp q)\oplus i\cdot\m$$
such that 
$$\sp{p+q}=(\sp p\oplus\sp q)\oplus\m$$ is the Lie algebra of the quaternionic unitary group $\Sp{p+q}$.

\begin{proposition}\label{proposition-Spp}
Let $z_{j\alpha},w_{j\alpha}:\Spp pq\to\mathbb{C}$ be the complex-valued matrix elements of the standard representation of the quaternionic unitary group $\Spp pq$. Then the tension field $\tau$ and the conformality operator $\kappa$ on $\Spp pq$ satisfy the following relations 
	$$
	\tau(z_{j\alpha})= -\,\frac{2(p+q)+1}2\cdot z_{j\alpha},
	\ \ \tau(w_{k \beta})=-\,\frac{2(p+q)+1}2\cdot w_{k \beta},
	$$
	$$
	\kappa(z_{j\alpha},z_{k\beta})
	=-\,\frac 12\cdot z_{k\alpha}\cdot z_{j\beta},\ \
	\kappa(w_{j\alpha},w_{k\beta})
	=-\,\frac 12\cdot w_{k\alpha}\cdot w_{j\beta},
	$$
	$$
	\kappa(z_{j\alpha},w_{k\beta})
	=-\,\frac 12\cdot z_{k\alpha}\cdot w_{j\beta}.
	$$
\end{proposition}

\begin{proof}
This is an immediate consequence of Lemma \ref{lemma-Sp}, the above relationship between the Lie algebras $\sp{p+q}$ and $\spp pq$ and how the semi-Riemannian metric $g$ is defined on the complex Lie algebra $\glh {p+q}$.
\end{proof}

\begin{theorem}\label{theorem-eigenfamily-Spp}
Let $u,v\in\cn^{p+q}$ be a non-zero element of $\cn^{p+q}$, then the complex $(p+q)$-dimensional vector space 
$$\E_{uv}=\{\phi_{ab}:\Spp pq\to\cn\ |\ 
\phi_{ab}(g)=\trace (u^taz^t+v^tbw^t),\ a,b\in \cn^{p+q}\}$$ is an eigenfamily on $\Spp pq$ such that for all $\phi,\psi\in\E_{uv}$ we have 
$$\tau(\phi)=-\,\frac{2(p+q)+1}2\cdot\phi,\ \ \kappa(\phi,\psi)=-\,\frac 12\cdot\phi\cdot\psi.$$
\end{theorem}

\begin{proof}
The statement follows directly from Proposition \ref{proposition-Spp}.
\end{proof}


\section{Acknowledgements}


The first author would like to thank the Department of Mathematics at Lund University for its great hospitality during her time there as a postdoc.

The authors would like to thank Albin Ingelstr\" om for useful discussions on case of $\SUU pq$.



\begin{thebibliography}{99}

\bibitem{Bai-Eel}
P. Baird, J. Eells,
{\it A conservation law for harmonic maps},
in Geometry Symposium Utrecht 1980,
Lecture Notes in Mathematics {\bf 894}, 1-25, Springer (1981).

\bibitem{Bai-Woo-book}
P. Baird, J.C. Wood,
{\it Harmonic morphisms between Riemannian manifolds},
The London Mathematical Society Monographs {\bf 29},
Oxford University Press (2003).

\bibitem{Fug-1}
B. Fuglede, 
{\it Harmonic morphisms between Riemannian manifolds},
Ann. Inst. Fourier {\bf 28} (1978), 107-144.

\bibitem{Fug-2}
B. Fuglede,
{\it Harmonic morphisms between semi-Riemannian manifolds},
Ann. Acad. Sci. Fenn. Math. {\bf 21} (1996), 31–50.

\bibitem{Gaz-Gru-Swe}
F. Gazzolla, H.-C. Grunau, G. Sweers,
{\it Polyharmonic boundary value problems},
Lecture Notes in Mathematics {\bf 1991},
Springer (2010).

\bibitem{Gha-Gud-2}
E. Ghandour, S. Gudmundsson,
{\it Explicit harmonic morphisms and $p$-harmonic functions on Grassmannians},
preprint (2020).

\bibitem{Gud-bib}
S. Gudmundsson,
{\it The Bibliography of Harmonic Morphisms},
{\tt www.matematik.lu.se/ matematiklu/personal/sigma/harmonic/bibliography.html}
	
\bibitem{Gud-p-bib}
S. Gudmundsson,
{\it The Bibliography of $p$-Harmonic Functions},
{\tt www.matematik.lu.se/ matematiklu/personal/sigma/harmonic/p-bibliography.html}

	
\bibitem{Gud-Mon-Rat-1}
S. Gudmundsson, S. Montaldo, A. Ratto,
{\it Biharmonic functions on the classical compact simple Lie groups},
J. Geom. Anal. {\bf 28} (2018), 1525-1547.
	
\bibitem{Gud-Sak-1}
S. Gudmundsson, A. Sakovich,
{\it Harmonic morphisms from the classical compact semisimple Lie groups},
Ann. Global Anal. Geom. {\bf 33} (2008), 343-356.
	
\bibitem{Gud-Sob-1}
S. Gudmundsson, M. Sobak,
{\it Proper r-harmonic functions from Riemannian manifolds},
Ann. Global Anal. Geom. {\bf 57} (2020), 217-223.

\bibitem{Hel}
S. Helgason,
{\it Differential Geometry, Lie Groups, and Symmetric Spaces}, Academic Press (1978).

\bibitem{Ish}
T. Ishihara, 
{\it A mapping of Riemannian manifolds which preserves harmonic functions}, 
J. Math. Kyoto Univ. {\bf 19} (1979), 215-229.

\bibitem{Kna}
A. W. Knapp, 
{\it Lie Groups Beyond an Introduction}, Progress in Mathematics {\bf 140}, Birkhäuser (2002).

\bibitem{Mel}
V. Meleshko,
{\it Selected topics in the history of the two-dimensional biharmonic problem},
Appl. Mech. Rev. {\bf 56} (2003), 33-85.

\bibitem{O-Nei}
B. O'Neill,
{\it Semi-Riemannian Geometry - With Applications to Relativity}, Pure and Applied Mathematics {\bf 103}, Academic Press (1983).

\end{thebibliography}
\end{document}